\documentclass{amsart}
\usepackage[T1]{fontenc}
\usepackage{latexsym,amssymb,amsmath}

\usepackage[all,ps]{xy}
\usepackage{amsthm}
\usepackage{longtable}
\usepackage{epsfig}
\usepackage{mathrsfs}
\usepackage{hhline}
\usepackage{epic}
\usepackage{enumerate}
\usepackage{graphicx}
\usepackage{xcolor}

\usepackage{tikz}

 \newcommand{\bt}{\boxtimes}

 \newcommand{\Alt}{\mathcal{A}}

 \newcommand{\sym}{\mathfrak{S}}
 \newcommand{\IBr}{\operatorname{IBr}}

 \newcommand{\Ind}{\operatorname{Ind}}
 \newcommand{\Res}{\operatorname{Res}}

 \newcommand{\C}{\mathbb{C}}

 \newcommand{\Z}{\mathbb{Z}}

 \newcommand{\Irr}{\operatorname{Irr}}

 \newcommand{\Br}{\operatorname{IBr}}

\newcommand{\cal}[1]{\mathcal{#1}}

 \newcommand{\dis}{\displaystyle}
 \newcommand{\la}{\lambda}

  \newcommand{\sa}{\sigma}

 \newcommand{\ga}{\gamma}

  \usepackage{xcolor}

\newtheorem{theorem}{Theorem}[section] 
\newtheorem{lemma}[theorem]{Lemma}     

\theoremstyle{definition}
\newtheorem{remark}[theorem]{Remark}

\title[]
{Restriction of characters to subgroups of wreath products and basic sets for the symmetric group}

\author{Jean-Baptiste Gramain}

\address{Institute of Mathematics, 
University of Aberdeen, King's College \\
Fraser Noble Building, Aberdeen AB24 3UE, UK
}
\email{jbgramain@abdn.ac.uk}
\author{Adriana Marciuk}

\address{National University of Singapore\\ Department of Mathematics\\ Faculty of Science\\ Level 4, Block S17\\ 10 Lower Kent Ridge Road\\ Singapore 119076\\ Singapore.}
\email{matmae@nus.edu.sg}


\begin{document}
\begin{abstract}
In this paper, we give the decomposition into irreducible characters of the restriction to the wreath product $\Z_{p-1} \wr \sym_w$ of any irreducible character of $(\Z_p \rtimes \Z_{p-1}) \wr \sym_w$, where $p$ is any odd prime, $w \geq 0$ is an integer, and $\Z_p$ and $\Z_{p-1}$ denote the cyclic groups of order $p$ and $p-1$ respectively. This answers the question of how to decompose the restrictions to $p$-regular elements of irreducible characters of the symmetric group $\sym_n$ in the $\Z$-basis corresponding to the $p$-basic set of $\sym_n$ described by Brunat and Gramain in \cite{BrGr}. The result is given in terms of the Littlewood-Richardson coefficients for the symmetric group.\end{abstract}
\maketitle

\section{Introduction}
\label{sec:intro}

Let $G$ be a finite group and $\Irr (G)$ be the set of irreducible complex characters of $G$. Let $p$ be a prime (dividing $ |G|$), and let ${\cal C}$ be the set of $p$-regular elements of $G$. For each $\chi \in \C \Irr(G)$, we define a class function $\chi^{\cal C}$ of $G$ by letting
$$\chi^{\cal C}(g)= \left\{ \begin{array}{ll} \chi(g) & \mbox{if} \; g \in {\cal C}, \\ 0 & \mbox{otherwise.} \end{array} \right. $$
One of the fundamental results of Brauer's Theory is the existence of a surjective homomorphism, called the decomposition homomorphism,
$$d \colon \left\{ \begin{array}{rcl} \Z \Irr (G) & \longrightarrow & \Z \IBr_p (G) \\ \chi & \longmapsto & \chi^{\cal C} \end{array} \right. ,$$
where $\IBr_p(G)$ is the set of irreducible ($p$-modular) Brauer characters of $G$. The matrix $D$ of $d$ in the bases $\Irr(G)$ and $\Br_p(G)$ is the {\emph{($p$-modular) decomposition matrix of $G$}}. Up to reordering the rows and columns, the matrix $D$ is diagonal by blocks, which gives partitions of $\Irr(G)$ and $\Br_p(G)$ into {\emph{$p$-blocks}}.

\smallskip

While finding the decomposition matrix of a group is a very difficult problem, {\emph{basic sets}} can sometimes help computing Brauer characters and/or the decomposition matrix $D$, or at least reduce the problem. We call {\emph{$p$-basic set for $G$}} any subset ${\cal B} \subset \Irr (G)$ such that the family $\cal B^{\cal C}=\{ \chi^{\cal C}, \, \chi \in \cal B \}$ is a $\Z$-basis for the $\Z$-module generated by $\Irr^{\cal C}(G) =\{ \chi^{\cal C}, \, \chi \in \Irr(G) \}$. In particular, $|\cal B|$ is the number of $p$-regular conjugacy classes of $G$. One can also define the notion of $p$-basic set for a $p$-block of $G$, and one shows easily that, if each $p$-block $b$ of $G$ has a $p$-basic set ${\cal B}_b$, then the union of the ${\cal B}_b$'s is a $p$-basic set for $G$.

If $\cal B$ is a $p$-basic set for $G$, and if we write $\chi^{\cal C}= \displaystyle \sum_{\psi \in \cal B} n_{\chi \psi} \cdot \psi^{\cal C}$ ($\chi \in \Irr(G), \, n_{\chi \psi} \in \Z$) and $N_{\cal B} = (( n_{\chi \psi} ))_{\chi \in \Irr(G), \, \psi \in \cal B}$, and $D_{\cal B}$ for the (square) sub-matrix of $D$ whose rows correspond to $\cal B$, then we have $D=N_{\cal B} D_{\cal B}$, so that computing the matrix $N_{\cal B}$ reduces the problem of finding $D$ to computing (the smaller matrix) $D_{\cal B}$.

In \cite{BrGr}, the authors describe, for any integer $n$ and odd prime $p$, a $p$-basic set ${\cal B}$ for the symmetric group $\sym_n$. The object of the present paper is, in this case, to describe completely the matrix $N_{\cal B}$. It should be noted that another $p$-basic set for $\sym_n$ was previously known (see \cite[Section 6.3]{James-Kerber}), but that ${\cal B}$ has further properties which allow it to restrict to a $p$-basic set for the alternating group $\Alt_n$.

\medskip
Throughout this paper, we let $n\geq 1$ be any integer, and $p$ be an odd prime. The irreducible complex characters of the symmetric group $\sym_n$ are canonically labelled by partitions of $n$, and we write $\Irr(\sym_n)=\{ \chi_{\la} \, | \, \la \vdash n \}$. For any $\la \vdash n$, we write $n=| \la |$, the {\emph{size}} of $\la$. The distribution of irreducible characters of $\sym_n$ into $p$-blocks is described by the Nakayama Conjecture (see \cite[6.1.21]{James-Kerber}). Each partition $\la$ of $n$ is completely and uniquely determined by its {\emph{$p$-core}} $\ga_p(\la)$ and its $p$-quotient $q_p(\la)$. The $p$-core $\ga_p(\lambda)$ is the partition, of some integer $s$, obtained by removing from $\la$ all the {\emph{hooks}} of length divisible by $p$, and the $p$-quotient $q_p(\la)$ is a $p$-tuple $(\lambda^1, \, \la^2, \, \ldots , \, \la^p)$ of partitions whose sizes add up to the integer $w$ (written $(\lambda^1, \, \la^2, \, \ldots , \, \la^p) \Vdash w$), called the {\emph{$p$-weight of $\la$}}, and such that $n = s+pw$. Then two characters $\chi_{\la}, \, \chi_{\mu} \in \Irr(\sym_n)$ belong to the same $p$-block of $\sym_n$ if and only if $\ga_p(\la)=\ga_p(\mu)$. In particular, if that is the case, then $\la$ and $\mu$ have the same $p$-weight, and characters in the $p$-block of $\sym_n$ corresponding to the $p$-core $\ga\vdash s$ are labelled by the $p$-tuples of partitions $(\nu^1, \, \nu^2, \, \ldots , \, \nu^p) \Vdash w$, where $w=(n-s)/p$, and the $p$-block is said to have $p$-weight $w$.

\smallskip
In \cite{BrGr}, the authors show that $\{\chi_{\la} \in \Irr(\sym_n) \, | \, q_p(\la)=( \lambda^1, \, \ldots , \, \la^p) \; \mbox{with} \; \la^{r}=\emptyset\}$ is a $p$-basic set for $\sym_n$, where $r=\frac{p+1}{2}$. To prove this, they construct, for each $w \geq 0$, a {\emph{generalized perfect isometry}} between the set of irreducible characters of a $p$-block $b$ of $p$-weight $w$ and the set of irreducible characters of the wreath product $(\Z_p \rtimes \Z_{p-1}) \wr \sym_w$, where $\Z_p$ and $\Z_{p-1}$ denote the cyclic groups of order $p$ and $p-1$ respectively. The irreducible characters of $(\Z_p \rtimes \Z_{p-1}) \wr \sym_w$ can be parametrized by the $p$-tuples $\nu=(\nu^1,  \, \ldots , \, \nu^p) \Vdash w$ of partitions of $w$ (see Section \ref{sec:IrrCharWreath}), in such a way that a character has the subgroup $\Z_p^w$ of $(\Z_p \rtimes \Z_{p-1})\wr \sym_w=\Z_p^w \rtimes (\Z_{p-1} \wr \sym_w)$ in its kernel if and only if it is labelled by $\nu=(\nu^1,  \, \ldots , \, \nu^p) \Vdash w$ such that $\nu^r=\emptyset$. We write $\Irr((\Z_p \rtimes \Z_{p-1}) \wr \sym_w)=\{\chi^{\nu} \, | \, \nu \Vdash w\}$. We also let ${\cal B}_b=\{ \chi_{\la} \in \Irr(b) \, | \, q_p(\la)=( \lambda^1, \, \ldots , \, \la^p) \; \mbox{with} \; \la^{r}=\emptyset\}$ be the $p$-basic set for $b$ constructed in \cite{BrGr}. The results of \cite{BrGr} show that there is an explicit bijection $\chi_{\la} \longmapsto \chi_{\tilde{\la}}$ from $\Irr(b)$ to itself, which restricts to the identity on ${\cal B}_b$, as well as explicitly determined signs $\{ \varepsilon (\la) \, , \, \chi_{\la} \in \Irr(b) \}$ such that, if we write $\chi_{\la}^{\cal C}= \displaystyle \sum_{\chi_{\mu} \in {\cal B}_b} n_{\la \mu}  \cdot \chi_{\mu}^{\cal C}$ ($\chi_{\la} \in \Irr(b), \, n_{\la \mu} \in \Z$), then the coefficients $n_{\la \mu}$ ($\chi_{\la} \in \Irr(b), \, \chi_{\mu} \in {\cal B}_b$) are given by
\begin{equation}
\label{eq:coef}n_{\la \mu}= \varepsilon(\la) \varepsilon(\mu) \langle \Res^{(\Z_p \rtimes \Z_{p-1})\wr \sym_w}_{\Z_{p-1} \wr \sym_w}(\chi^{q_p(\tilde{\la})}), \Res^{(\Z_p \rtimes \Z_{p-1})\wr \sym_w}_{\Z_{p-1} \wr \sym_w}(\chi^{q_p(\tilde{\mu})}) \rangle_{\Z_{p-1} \wr \sym_w}.\end{equation}
Note that, by construction, for any $\chi_{\mu} \in {\cal B}_b$, we have $q_p(\tilde{\mu})=(\nu^1,  \, \ldots , \, \nu^p)$ with $\nu^r=\emptyset$. Thus $\chi^{q_p(\tilde{\mu})}$ has $\Z_p^w$ in its kernel, and $\Res^{(\Z_p \rtimes \Z_{p-1})\wr \sym_w}_{\Z_{p-1} \wr \sym_w}(\chi^{q_p(\tilde{\mu})})$ is actually an irreducible character of ${\Z_{p-1} \wr \sym_w}$. Hence, in order to decompose any restriction to $p$-regular elements of an irreducible character of $\sym_n$ as a $\Z$-linear combination of the restrictions of characters in the basic set ${\cal B}$, it is sufficient to compute the decomposition into irreducible characters of $\Z_{p-1} \wr \sym_w$ of any irreducible character of $(\Z_p \rtimes \Z_{p-1}) \wr \sym_w$. This decomposition is given by our main result, Theorem \ref{thm:main}.

\medskip

The paper is organised as follows. In Section \ref{sec:WreathProducts}, we recall classical results about the conjugacy classes and irreducible complex characters of wreath products. These results are then applied to the groups $(\Z_p \rtimes \Z_{p-1}) \wr \sym_w$ and $\Z_{p-1} \wr \sym_w$ in Section \ref{sec:IrrCharWreath}, and the irreducible characters of these groups are parametrized in ways that are compatible (see Theorem \ref{thm:Compatible}). In Section \ref{sec:SpecialCase}, we describe the characters of $(\Z_p \rtimes \Z_{p-1}) \wr \sym_w$ induced by some specific characters of $\Z_{p-1} \wr \sym_w$. These particular cases form the basis for the computations of Section \ref{sec:Coef}, where we explicitly decompose into irreducibles the induction to $(\Z_p \rtimes \Z_{p-1}) \wr \sym_w$ of any irreducible character of $\Z_{p-1} \wr \sym_w$ (see Theorem \ref{thm:main}). This in turn provides a formula for any of the scalar products appearing in Equation (\ref{eq:coef}).

\section{Conjugacy classes and irreducible characters of wreath products}
\label{sec:WreathProducts}

Throughout this section, we let $N$ be a finite group and $w\geq 1$ be an integer, and consider the wreath product $N \wr \sym_w$. That is, $N \wr \sym_w$ is the semidirect product $N^w
\rtimes \sym_w$, where $\sym_w$ acts by permutation on the $w$ copies
of $N$. For a complete description of wreath products and their representations,
we refer to \cite[Chapter 4]{James-Kerber}.

Let $s$ be the number of conjugacy classes of $N$, and let $g_1, \, \ldots , \, g_s$ be
representatives for the conjugacy classes of $N$. Then the conjugacy classes of $N \wr \sym_w$ can be parametrized by the
$s$-tuples of partitions of $w$ as follows. The elements of
$N \wr \sym_w$ are of the form $(h; \, \sigma)=((h_1, \, \ldots, \, h_w ) ; \, \sigma)$, with $h_1, \, \ldots, \, h_w \in N$ and $ \sigma \in \sym_w$. For any such element, write $\sigma=\sigma_1 * \cdots * \sigma_{c(\sigma)}$, a product of disjoint cycles. Then, for any $1 \leq \nu \leq c(\sigma)$, we have $\sigma_{\nu}=(j_{\nu}, j_{\nu} \sigma_{\nu}, \ldots,
j_{\nu} \sigma_{\nu}^{k_{\nu}-1})$ (where $\sigma_{\nu}$ is a $k_{\nu}$-cycle), and we define the {\emph{$\nu$-th cycle product of
$(h; \, \sigma)$}} by
$$
g_{\nu}(h; \, \sigma)= h_{j_{\nu}} \cdot h_{j_{\nu} \sigma_{\nu}^{-1}}  \cdot h_{j_{\nu} \sigma_{\nu}^{-2}} \cdot \, \cdots
 \, \cdot h_{j_{\nu} \sigma_{\nu}^{-(k_{\nu}-1)}}.
$$
In particular, $g_{\nu}(h; \, \sigma) \in N$. We then we form $s$ partitions $(\pi_1, \,
\ldots, \, \pi_s)$ as follows: each $1 \leq \nu \leq c(\sigma)$ gives a cycle of length $k_{\nu}$ in $\pi_i$ if the cycle product $g_{\nu}(h ;
\, \sigma)$ is conjugate to $g_i$ in $N$. The resulting $s$-tuple
of partitions of $w$ describes the {\emph{cycle structure}} of $(h; \,
\sigma)$, and two elements of $N \wr \sym_w$ are conjugate if and only if they
have the same cycle structure.

\smallskip
The irreducible complex characters of $N \wr \sym_w$ can also be
parametrized by the $s$-tuples of partitions of $w$ as follows. Let $\Irr(N)=\{ \omega_1, \, \ldots , \,
\omega_s \}$. Take any $\alpha=(\alpha^1, \, \ldots , \, \alpha^s) \Vdash w$ and consider
the irreducible character $\prod_{i=1}^s \omega_i^{|\alpha^i|}$ of the
{\emph{base group}} $N^w$. It can be extended in a natural way to its
inertia subgroup $N \wr \sym_{|\alpha^1|} \times \cdots \times N \wr
\sym_{|\alpha^s|}$, giving the irreducible character $\prod_{i=1}^s
\widetilde{\omega_i^{|\alpha^i|}}$. For each $1 \leq i \leq s$, the irreducible character $\widetilde{\omega_i^{|\alpha^i|}}$ of $N \wr \sym_{|\alpha^i|}$ is given as follows: if $(f; \, \pi) \in N \wr \sym_{|\alpha^i|}$ has cycle products $g_{\nu}(f ; \, \pi)$ ($1 \leq \nu \leq c(\pi)$), then $\widetilde{\omega_i^{|\alpha^i|}}(f; \, \pi)=\prod_{\nu=1}^{c(\pi)} \omega_i(g_{\nu}(f ; \, \pi))$ (see \cite[Lemma 4.3.9]{James-Kerber}). Any
extension of $\prod_{i=1}^s \omega_i^{|\alpha^i|}$ to $N \wr \sym_{|\alpha^1|} \times \cdots \times N \wr
\sym_{|\alpha^s|}$ is of the form $\Omega^{\alpha}=\prod_{i=1}^s (\widetilde{\omega_i^{|\alpha^i|}}
\otimes \Upsilon_{\alpha_i})$, where, for each $1 \leq i \leq s$, $\Upsilon_{\alpha_i} \in \C \Irr(N \wr \sym_{|\alpha^i|})$ is defined by $\Upsilon_{\alpha_i} (f ; \, \pi) = \chi_{\alpha_i} (f ; \, \pi)$ for all $(f ; \, \pi) \in  N \wr \sym_{|\alpha^i|}$ (and $\chi_{\alpha_i} \in \Irr(\sym_{|\alpha^i|})$, see \cite[4.3.15]{James-Kerber}). Then $\aleph^{\alpha}:=
\Ind_{\scriptstyle \prod_{i=1}^s N \wr \sym_{|\alpha^i|} }^{N \wr
\sym_w} ( \Omega^{\alpha} ) \in \Irr(N \wr \sym_w)$. Different $\alpha \Vdash w$
give different irreducible characters of $N \wr \sym_w$, and any irreducible
character of $N \wr \sym_w$ can be obtained in this way (see \cite[Theorem 4.3.34]{James-Kerber}).

\section{Parametrizations of $\Irr((\Z_p \rtimes \Z_{p-1}) \wr \sym_w)$ and $\Irr( \Z_{p-1} \wr \sym_w)$}
\label{sec:IrrCharWreath}

From now on, and throughout the paper, we fix an odd prime $p$, and we let $r=\frac{p+1}{2}$. We write $I$ for the set $\{1, \, 2, \, \ldots , \, p\} \setminus \{r\}$. We let $H=\Z_{p-1}$ and $G=\Z_p \rtimes \Z_{p-1}$, and, for any integer $k \geq 1$, we let $H_k=H \wr \sym_k$ and $G_k = G \wr \sym_k$.

We start by describing the irreducible characters of $G$ and their restrictions to $H$. The irreducible complex characters of $G$ are described as follows. We have $\Irr(G)=\{ \psi_1, \, \ldots , \,
\psi_p \}$, with $\psi_i(1)=1$ for $i \in I$,
and $\psi_r(1)=p-1$. More precisely, writing $\eta_1=1_{\mathbb{Z}_p}$ for the
trivial character of $\mathbb{Z}_p$, we have
$$
\Ind_{\mathbb{Z}_p}^G(1_{\mathbb{Z}_p})= \displaystyle \sum_{i \in I} \psi_i \; \; \mbox{and} \; \;
\Res_{\mathbb{Z}_p}^G(\psi_i)= 1_{\mathbb{Z}_p} \; \; (i \in I),
$$
and
$$
\Res_{\mathbb{Z}_p}^G(\psi_r)= \eta_2 + \, \cdots \, + \eta_p \; \;
\mbox{and} \; \; \Ind_{\mathbb{Z}_p}^G(\eta_i)= \psi_r \; \; (2 \leq i
\leq p),
$$
where $\{ \eta_2 , \, \ldots , \, \eta_p \} = \Irr (\mathbb{Z}_p)
\setminus \{ 1_{\mathbb{Z}_p} \}$.

\smallskip
Now write $\Irr(H)=\{\theta_i \, | \, i \in I\}= \{ \theta_1=1_H, \, \theta_2, \, \ldots , \, \theta_{r-1} , \, \theta_{r+1} , \, \ldots , \, \theta_p\}$. For any $i \in I$, $\psi_i$ has $\Z_p \triangleleft G$ in its kernel. Thus, without loss of generality, we can choose the labelling such that $\psi_i = \theta_i \circ \varpi$, where $\varpi \colon G=\Z_p \rtimes \Z_{p-1} \longrightarrow \Z_{p-1}=H$ is the canonical surjection. In particular,
\begin{equation}
\label{eq:Restrictpsii}
\Res^G_H(\psi_i)=\theta_i \qquad \mbox{for all} \; i \in I.
\end{equation}
To describe $\Res^G_H(\psi_r)$, we start by noticing that $\psi_r(g)=0$ for all $g \in G \setminus \Z_p$. Indeed, by the first orthogonality relation, we have
$$\dis \sum_{g \in \Z_p} \psi_r(g) \overline{\psi_r(g)}=\sum_{2 \leq i,j \leq p} \sum_{g \in \Z_p} \eta_i(g) \overline{\eta_j(g)}=\sum_{2 \leq i \leq p} \sum_{g \in \Z_p} \eta_i(g) \overline{\eta_i(g)}=p(p-1).$$
Since also $ \dis \sum_{g \in G} \psi_r(g) \overline{\psi_r(g)}=|G|=p(p-1)$, this yields $$\dis \sum_{g \in G \setminus \Z_p} \psi_r(g) \overline{\psi_r(g)}=\sum_{g \in G \setminus \Z_p} |\psi_r(g)|^2=0,$$so that $\psi_r(g)=0$ for all $g \in G \setminus \Z_p$. Since $\psi_r(1)=p-1$, this gives 
\begin{equation}
\label{eq:respsir}
\Res^G_H(\psi_r)=(p-1) \delta_1,
\end{equation}
where $\delta_1$ is the indicator function of (the conjugacy class in $H$ of) 1. Now, by the second orthogonality relation, we have, for any $h \in H$
$$\dis \sum_{i \in I} \theta_i(h)=\dis \sum_{i \in I} \theta_i(h)\theta_i(1)=|C_H(1)|\delta_1(h)=(p-1)\delta_1(h).$$
Hence
\begin{equation}
\label{eq:Restrictpsir}
\Res^G_H(\psi_r)=\dis \sum_{i \in I} \theta_i .
\end{equation}

\medskip
If we now take any integer $w \geq 1$, then the irreducible complex characters of $G_w$ and $H_w$ are constructed as in Section \ref{sec:WreathProducts}.

\smallskip
\noindent
The irreducible complex characters of $G_w=G \wr \sym_w$ are
parametrized by the $p$-tuples of partitions of $w$ as follows. For
any $\alpha=(\alpha^1, \, \ldots , \, \alpha^p) \Vdash w$, $\chi^{\alpha} \in \Irr(G_w)$ is given by
$$\chi^{\alpha}= \Ind_{ \prod_{1 \leq i \leq p} 
G_{|\alpha^i|}}^{G_w} (\Psi^{\alpha})=\Ind_{\prod_{1 \leq i \leq p} G_{|\alpha^i|}}^{G_w} (\dis \prod_{1 \leq i \leq p} \widetilde{\psi_i^{|\alpha^i|}}
\otimes \varphi_{\alpha^i}),$$
where, for any $1 \leq i \leq p$ and $(f; \, \pi) \in G_{| \alpha^i|}$ with cycle products $g_{\nu}(f ; \, \pi)$ ($1 \leq \nu \leq c(\pi)$), we have $\widetilde{\psi_i^{|\alpha^i|}}(f; \, \pi)=\prod_{\nu=1}^{c(\pi)} \psi_i(g_{\nu}(f ; \, \pi))$ and $\varphi_{\alpha^i}(f; \, \pi)=\chi_{\alpha^i}(\pi)$.

\smallskip
\noindent
Note for future reference that, in the above notation, if
$\alpha^r = \emptyset$, then $\mathbb{Z}_p^w \subseteq \ker
(\chi^{\alpha})$. Indeed, for all $1 \leq i \leq p$, $i \neq r$,
we have $ \Res_{\mathbb{Z}_p}^G(\psi_i)= 1_{\mathbb{Z}_p}$; thus
$\Res_{\mathbb{Z}_p^{|\alpha^i|}}^{G^{|\alpha^i|}}(\psi_i^{|\alpha^i|})=
1_{\mathbb{Z}_p^{|\alpha^i|}}$, so that
$\widetilde{\psi_i^{|\alpha^i|}}(g)=\widetilde{\psi_i^{|\alpha^i|}}(1)$
for all $g \in \mathbb{Z}_p^{|\alpha^i|}$,
and $(\widetilde{\psi_i^{|\alpha^i|}} \otimes
\varphi_{\alpha_i})(g)=(\widetilde{\psi_i^{|\alpha^i|}} \otimes
\varphi_{\alpha_i})(1)$ for all $g \in \mathbb{Z}_p^{|\alpha^i|}$. Since
$\Z_p^w \leq G^w \triangleleft G \wr \sym_w$, we easily get that, for all $g
\in \mathbb{Z}_p^w$, $\chi^{\alpha}(g)=\chi^{\alpha}(1)$. In particular, if
$\alpha^r = \emptyset$, then $\chi^{\alpha}=\xi \circ \overline{\varpi}$ for some $\xi \in \Irr(H_w)$, where $\overline{\varpi} \colon G_w=(\Z_p)^w \rtimes (\Z_{p-1} \wr \sym_w) \longrightarrow \Z_{p-1} \wr \sym_w=H_w$ is the canonical surjection.

\medskip
\noindent
The irreducible complex characters of $H_w=H \wr \sym_w$ are
parametrized by the $(p-1)$-tuples of partitions of $w$ as follows. For
any $\alpha=(\alpha^1, \, \ldots , \, \alpha^{r-1}, \, \alpha^{r+1} , \, \ldots , \,  \alpha^p) \Vdash w$, $\xi^{\alpha} \in \Irr(H_w)$ is given by
$$\xi^{\alpha}= \Ind_{ \prod_{i \in I} 
H_{|\alpha^i|}}^{H_w} (\Theta^{\alpha})=\Ind_{\prod_{i \in I} H_{|\alpha^i|}}^{H_w} (\dis \prod_{i \in I} \widetilde{\theta_i^{|\alpha^i|}}
\otimes \zeta_{\alpha^i}),$$
where, for any $i \in I$ and $(f; \, \pi) \in H_{| \alpha^i|}$ with cycle products $g_{\nu}(f ; \, \pi)$ ($1 \leq \nu \leq c(\pi)$), we have $\widetilde{\theta_i^{|\alpha^i|}}(f; \, \pi)=\prod_{\nu=1}^{c(\pi)} \theta_i(g_{\nu}(f ; \, \pi))$ and $\zeta_{\alpha^i}(f; \, \pi)=\chi_{\alpha^i}(\pi)$.

\medskip
The following result will be useful when we next consider the restriction to $H_w$ of some irreducible characters of $G_w$.

\begin{lemma}
\label{lem:easyrestrictions}
For any integer $k \geq 1$, $i \in I$ and $\la \vdash k$, we have
$$\Res^{G_k}_{H_k}(\widetilde{\psi_i^k})=\widetilde{\theta_i^k} \qquad \mbox{and} \qquad \Res^{G_k}_{H_k}(\varphi_{\la})=\zeta_{\la}.$$

\end{lemma}
\begin{proof}
Take any $(f; \, \pi) \in H_k$ (i.e. $(f ; \, \pi)=(1, \, (f ; \, \pi))\in \Z_p^k \rtimes H_k=G_k$, where $f\in H^k$ and $\pi \in \sym_k$), and let $g_{\nu}(f; \, \pi)$ ($1 \leq \nu \leq c(\pi)$) be the cycle products of $(f; \, \pi)$. Note that, for all $1 \leq \nu \leq c(\pi)$, $g_{\nu}(f; \, \pi) \in H$ (as a product of elements of $H$). Since, by (\ref{eq:Restrictpsii}), $\Res^G_H(\psi_i)=\theta_i$, this yields
$$\widetilde{\psi_i^k}(f; \, \pi)= \dis \prod_{\nu=1}^{c(\pi)} \psi_i ( g_{\nu}(f; \, \pi)) = \dis \prod_{\nu=1}^{c(\pi)} \theta_i ( g_{\nu}(f; \, \pi)) = \widetilde{\theta_i^k}(f; \, \pi),$$as claimed. The second part is immediate, as, for any $(f; \, \pi) \in H_k$, we have $\varphi_{\la}(f; \, \pi)=\chi_{\la}(\pi)=\zeta_{\la}(f; \, \pi)$.

\end{proof}

We will also need the following in Section \ref{sec:SpecialCase}.
\begin{lemma}
\label{lem:easyrestrictionsr}
For any integer $k \geq 1$ and any $(f; \, \pi) \in H_k$ with cycle products $g_{\nu}(f; \, \pi)$ ($1 \leq \nu \leq c(\pi)$), we have
$$\widetilde{\psi_r^k}(f; \, \pi)= \left\{ \begin{array}{cl} (p-1)^{c(\pi)} & \mbox{if $g_{\nu}(f; \, \pi)=1$ for all $1 \leq \nu \leq c(\pi)$}, \\ 0 & \mbox{otherwise}.\end{array} \right.$$
\end{lemma}
\begin{proof}
Take any $(f; \, \pi) \in H_k$ as in the proof of Lemma \ref{lem:easyrestrictions} and again note that, for all $1 \leq \nu \leq c(\pi)$, $g_{\nu}(f; \, \pi) \in H$. We therefore have
$$\widetilde{\psi_r^k}(f; \, \pi)= \dis \prod_{\nu=1}^{c(\pi)} \psi_r(g_{\nu}(f; \, \pi)) = \dis \prod_{\nu=1}^{c(\pi)} \Res^G_H(\psi_r)(g_{\nu}(f; \, \pi)).$$
The result immediately follows from (\ref{eq:respsir}).
\end{proof}

We can now show how our parametrizations for $\Irr(G_w)$ and $\Irr(H_w)$ are related.

\begin{theorem}
\label{thm:Compatible}
Take any $(\alpha^1, \, \ldots , \alpha^{r-1}, \, \emptyset , \, \alpha^{r+1}, \, \ldots \alpha^p) \Vdash w$.

\noindent
If we let $\alpha=(\alpha^1, \, \ldots , \alpha^{r-1}, \, \alpha^{r+1}, \, \ldots \alpha^p)$ and $\hat{\alpha}=(\alpha^1, \, \ldots , \alpha^{r-1}, \, \emptyset , \, \alpha^{r+1}, \, \ldots \alpha^p)$, then we have
$$\Res^{G_w}_{H_w} \left(\chi^{\hat{\alpha}}\right) = \xi^{\alpha}.$$
\end{theorem}

\begin{proof}
Let $\alpha=(\alpha^1, \, \ldots , \alpha^{r-1}, \, \alpha^{r+1}, \, \ldots \alpha^p)$ and $\hat{\alpha}=(\alpha^1, \, \ldots , \alpha^{r-1}, \, \emptyset , \, \alpha^{r+1}, \, \ldots \alpha^p)$ be as above. Then
$$\chi^{\hat{\alpha}}= \Ind_{ \prod_{1 \leq i \leq p} 
G_{|\alpha^i|}}^{G_w} (\Psi^{\hat{\alpha}})=\Ind_{ \prod_{i \in I} 
G_{|\alpha^i|}}^{G_w} (\Psi^{\alpha}),$$
where $\Psi^{\alpha}=\dis \prod_{i \in I} \widetilde{\psi_i^{|\alpha^i|}}
\otimes \varphi_{\alpha^i}$.

Let $\sa_1, \, \ldots , \, \sa_m$ be left coset representatives for the Young subgroup $\prod_{i \in I} \sym_{|\alpha^i|}$ of $\sym_w$. In particular, $(1; \, \sa_1), \, \ldots , \, (1; \, \sa_m)$ are also left coset representatives for $\prod_{i \in I} H_{|\alpha^i|}$ in $H_w$, and for $\prod_{i \in I} G_{|\alpha^i|}$ in $G_w$ (where, in the first instance, $(1; \, \sa_i)=(1_{H}; \, \sa_i)$ for $1 \leq i \leq m$ and, in the second, $(1; \, \sa_i)=(1_{G}; \, \sa_i)$ for $1 \leq i \leq m$). Since it will always be clear which group we are working in, we will denote all these coset representatives as $\sa_1, \, \ldots , \, \sa_m$.

By the formula for character induction, for all $g \in H_w$, we have $$\chi^{\hat{\alpha}}(g)=\dis \sum_{k=1}^m \overbrace{\Psi^{\alpha}}(\sa_k g \sa_k^{-1}),$$where
$$\overbrace{\Psi^{\alpha}}(\sa_k g \sa_k^{-1}) = \left\{ \begin{array}{ccc} 0 & \mbox{if} & \sa_k g \sa_k^{-1} \not\in \dis \prod_{i \in I} 
G_{|\alpha^i|}, \\ \Psi^{\alpha}(\sa_k g \sa_k^{-1}) & \mbox{if} & \sa_k g \sa_k^{-1} \in \dis \prod_{i \in I} 
G_{|\alpha^i|}. \end{array} \right.$$
Now, if $g=(f; \, \rho) \in G_w$ (with $f \in H^w \leq G^w$), then, for all $1 \leq k \leq m$, we have (see \cite[4.2.6]{James-Kerber}),
$$  \sa_k g \sa_k^{-1} = (1; \, \sa_k) (f ; \, \rho) (1; \, \sa_k)^{-1}= (f_{\sa_k}; \, \sa_k \rho \sa_k^{-1}),$$
where, if $f=((1,f_1), \, (1, f_2), \, \ldots , \, (1,f_w)) \in G^w$ (with $(f_1, \, \ldots , \, f_w) \in H^w$), then $f_{\sa_k}=((1,f_{\sa_k^{-1}(1)}), \, (1, f_{\sa_k^{-1}(2)}), \, \ldots , \, (1,f_{\sa_k^{-1}(w)}))$. In particular, $f_{\sa_k} \in H^w$, and $\sa_k g \sa_k^{-1}= (f_{\sa_k}; \, \sa_k \rho \sa_k^{-1}) \in H_w$. Hence, for all $1 \leq k \leq m$, we have $\sa_k g \sa_k^{-1} \in \dis \prod_{i \in I} 
G_{|\alpha^i|}$ if and only if $\sa_k g \sa_k^{-1} \in \dis \prod_{i \in I} 
H_{|\alpha^i|}$. If that is the case, then we can write $\sa_k g \sa_k^{-1}= \prod_{i \in I} (g_i ; \, \rho_i)$, where $(g_i ; \, \rho_i) \in H_{|\alpha^i|}$ for all $i \in I$, and, still in that case, we obtain
$$\begin{array}{rcl} \Psi^{\alpha}(\sa_k g \sa_k^{-1}) & = & \dis \prod_{i \in I} \left[ \widetilde{\psi_i^{|\alpha^i|}}
\otimes \varphi_{\alpha^i} \right] (g_i ; \, \rho_i) \\  & = & \dis \prod_{i \in I} \widetilde{\psi_i^{|\alpha^i|}} (g_i ; \, \rho_i) 
 \cdot \varphi_{\alpha^i} (g_i ; \, \rho_i) \\ & = & \dis \prod_{i \in I} \Res^{G_{|\alpha^i|}}_{H_{|\alpha^i|}} \left(\widetilde{\psi_i^{|\alpha^i|}}\right) (g_i ; \, \rho_i) \cdot
\Res^{G_{|\alpha^i|}}_{H_{|\alpha^i|}} \left( \varphi_{\alpha^i} \right) (g_i ; \, \rho_i) \\ 
 & = & \dis \prod_{i \in I} \widetilde{\theta_i^{|\alpha^i|}} (g_i ; \, \rho_i) \cdot  \zeta_{\alpha^i} (g_i ; \, \rho_i) \qquad \mbox{(by Lemma \ref{lem:easyrestrictions})} \\ 
 & = & \dis \prod_{i \in I} \left[ \widetilde{\theta_i^{|\alpha^i|}} \otimes \zeta_{\alpha^i} \right] (g_i ; \, \rho_i) \\
 & = & \left[ \dis \prod_{i \in I} \widetilde{\theta_i^{|\alpha^i|}} \otimes \zeta_{\alpha^i} \right] (\sa_k g \sa_k^{-1}) \\
 & = & \Theta^{\alpha}(\sa_k g \sa_k^{-1}) .
 \end{array}$$

We therefore have, for any $1 \leq k \leq m$,
$$\overbrace{\Psi^{\alpha}}(\sa_k g \sa_k^{-1}) = \left\{ \begin{array}{ccc} 0 & \mbox{if} & \sa_k g \sa_k^{-1} \not\in \dis \prod_{i \in I} 
H_{|\alpha^i|}, \\ \Theta^{\alpha}(\sa_k g \sa_k^{-1}) & \mbox{if} & \sa_k g \sa_k^{-1} \in \dis \prod_{i \in I} 
H_{|\alpha^i|}. \end{array} \right.$$
Since $\sa_1, \, \ldots , \, \sa_m$ are also left coset representatives for $\prod_{i \in I} H_{|\alpha^i|}$ in $H_w$, we get, for all $g \in H_w$,
$$\chi^{\hat{\alpha}}(g)=\dis \sum_{k=1}^m \overbrace{\Psi^{\alpha}}(\sa_k g \sa_k^{-1})=\dis \sum_{k=1}^m \overbrace{\Theta^{\alpha}}(\sa_k g \sa_k^{-1})=\Ind_{ \prod_{i \in I} 
H_{|\alpha^i|}}^{H_w} (\Theta^{\alpha})(g)=\xi^{\alpha}(g),$$
as claimed.

\end{proof}

\begin{remark}
\label{rem:Compatible}
In view of the observation we made when parametrizing the irreducible characters of $G_w$, the statement of Theorem \ref{thm:Compatible} can be rephrased as: if $\chi \in \Irr(G_w)$ is labelled by $(\alpha^1, \, \ldots , \alpha^{r-1}, \, \emptyset , \, \alpha^{r+1}, \, \ldots \alpha^p)$, then $\chi=\xi \circ \overline{\varpi}$, where $\xi \in \Irr(H_w)$ is labelled by $(\alpha^1, \, \ldots , \alpha^{r-1}, \alpha^{r+1}, \, \ldots \alpha^p)$ and $\overline{\varpi} \colon G_w \longrightarrow H_w$ is the canonical surjection.

\end{remark}

\section{Induction of some special characters}
\label{sec:SpecialCase}
In this section, we fix any integer $k \geq 1$, and we will describe the induced characters $\Ind^{G_k}_{H_k}\left(\widetilde{\theta_i^k} \otimes \zeta_{\alpha}\right)$ for $i \in I$ and $\alpha \vdash k$ (see Theorem \ref{thm:inducedthetaialpha}). We start by some results on multiplicities.

\begin{lemma}
\label{lem:multiplicitythetai}
For any $i \in I$, the multiplicity of the irreducible character $\widetilde{\theta_i^k}$ in $\Res^{G_k}_{H_k}\left(\widetilde{\psi_r^k}\right)$ is 1.
\end{lemma}
\begin{proof} Take any $i \in I$, and let $A_{i,r}$ be the multiplicity of $\widetilde{\theta_i^k}$ in $\Res^{G_k}_{H_k}\left(\widetilde{\psi_r^k}\right)$. Then
$$\begin{array}{rcl} A_{i,r} & = & \dis \frac{1}{|H_k|} \sum_{(f ; \, \pi) \in H_k} \widetilde{\psi_r^k}(f ; \, \pi) \cdot \overline{\widetilde{\theta_i^k}(f ; \, \pi)} \\ & = & \dis \frac{1}{|H_k|} \sum_{\pi \in \sym_k} \sum_{f \in H^k \; \mbox{\tiny{such that}} \atop g_{\nu}(f; \, \pi)=1 \, \mbox{\tiny{for all}} \; 1 \leq \nu \leq c(\pi)} (p-1)^{c(\pi)} \cdot \overline{\widetilde{\theta_i^k}(f ; \, \pi)} \quad \mbox{(by Lemma \ref{lem:easyrestrictionsr}).}
\end{array}$$
And, whenever $(f ; \, \pi) \in H_k$ is such that $g_{\nu}(f; \, \pi)=1$ for all $1 \leq \nu \leq c(\pi)$, we have
$$\widetilde{\theta_i^k}(f ; \, \pi)=\dis \prod_{\nu=1}^{c(\pi)} \theta_i ( g_{\nu}(f; \, \pi)) = \dis \prod_{\nu=1}^{c(\pi)}1 = 1.$$
Hence
$$\begin{array}{rcl} A_{i,r} & = & \dis \frac{1}{(p-1)^k \cdot k!} \sum_{\pi \in \sym_k} \sum_{f \in H^k \; \mbox{\tiny{such that}} \atop g_{\nu}(f; \, \pi)=1 \, \mbox{\tiny{for all}} \; 1 \leq \nu \leq c(\pi)} (p-1)^{c(\pi)} \\ & = &\dis \frac{1}{(p-1)^k \cdot k!} \sum_{\pi \in \sym_k}(p-1)^{c(\pi)} \cdot |{\cal G}(\pi)|, \end{array}$$
where, for any $\pi \in \sym_k$, ${\cal G}(\pi)= \{ f \in H^k \, | \, g_{\nu}(f; \, \pi)=1 \; \mbox{for all} \;  1 \leq \nu \leq c(\pi)\}$. 

Now, if we write $\pi=\pi_1 * \cdots * \pi_{c(\pi)}$, a product of disjoint cycles, we see that $f \in {\cal G}(\pi)$ if and only if, after reordering the ``coordinates'' of $f$ according to the cycles of $\pi$, $f$ is of the form $(f_1, \, \ldots , \, f_{c(\pi)})$, where each $f_{\nu}$ is a $|\pi_{\nu}|$-tuple of elements of (the abelian group) $H$ whose product is $g_{\nu}(f, \pi)=1$. So, for each $1 \leq \nu \leq c(\pi)$, we can choose the first $|\pi_{\nu}|-1$ coordinates of $f_{\nu}$ to be anything we want in $H$, and the last coordinate is imposed by the condition $g_{\nu}(f, \pi)=1$. This means that, for each $1 \leq \nu \leq c(\pi)$, we have $(p-1)^{|\pi_{\nu}|-1}$ choices for $f_{\nu}$, so that $$|{\cal G}(\pi)|= \dis \prod_{\nu=1}^{c(\pi)} (p-1)^{|\pi_{\nu}|-1}.$$
We therefore have
 

$$\begin{array}{rcl} (p-1)^k \cdot k! \cdot A_{i,r} & = & \dis \sum_{\pi \in \sym_k}(p-1)^{c(\pi)} \cdot \prod_{\nu=1}^{c(\pi)} (p-1)^{|\pi_{\nu}|-1} =  \dis \sum_{\pi \in \sym_k} \prod_{\nu=1}^{c(\pi)} (p-1)^{|\pi_{\nu}|}\\ & & \\
 & = & \dis \sum_{\pi \in \sym_k}  (p-1)^{\sum_{\nu=1}^{c(\pi)}|\pi_{\nu}|} =  \dis \sum_{\pi \in \sym_k}  (p-1)^k =  \dis (p-1)^k \cdot k!
 
\end{array}$$
and $A_{i,r}=1$, as claimed.

\medskip
\noindent
Note that our argument shows that, for any $\pi \in \sym_k$,
\begin{equation}
\label{eq:sumpsir}
\sum_{f  \in H^k} \widetilde{\psi_r^k}(f ; \, \pi) \cdot \overline{\widetilde{\theta_i^k} (f ; \, \pi)}=(p-1)^k.
\end{equation}
We will use this to prove the next result.

\end{proof}
\noindent
Lemma \ref{lem:multiplicitythetai} can be extended as follows:
\begin{lemma}
\label{lem:multiplicitythetaibeta}
For any $i \in I$ and any partitions $\alpha$ and $\beta$ of $k$, the multiplicity of the irreducible character $\widetilde{\theta_i^k} \otimes \zeta_{\alpha}$ in $\Res^{G_k}_{H_k}\left(\widetilde{\psi_r^k} \otimes \varphi_{\beta}\right)$ is $\delta_{\alpha,\beta}$.
\end{lemma}
\begin{proof} 
Take any $i \in I$ and any partitions $\alpha$ and $\beta$ of $k$, and let $B_{i,r,\alpha,\beta}$ be the multiplicity of $\widetilde{\theta_i^k} \otimes \zeta_{\alpha}$ in $\Res^{G_k}_{H_k}\left(\widetilde{\psi_r^k} \otimes \varphi_{\beta}\right)$. Then
$$\begin{array}{rcl} B_{i,r,\alpha,\beta} & = & \dis \frac{1}{|H_k|} \sum_{(f ; \, \pi) \in H_k} \widetilde{\psi_r^k}(f ; \, \pi) \cdot \overline{\widetilde{\theta_i^k}(f ; \, \pi)} \cdot \varphi_{\beta}(f ; \, \pi) \cdot \overline{ \zeta_{\alpha}(f ; \, \pi)}\\ & = &  \dis \frac{1}{|H_k|} \sum_{\pi \in \sym_k} \chi_{\beta}(\pi) \overline{\chi_{\alpha}(\pi)} \cdot \left( \sum_{f \in H^k} \widetilde{\psi_r^k}(f ; \, \pi) \cdot \overline{\widetilde{\theta_i^k}(f ; \, \pi)}\right)\\ 
& = &  \dis \frac{1}{|H_k|} \sum_{\pi \in \sym_k} \chi_{\beta}(\pi) \overline{\chi_{\alpha}(\pi)} \cdot (p-1)^k
\quad \mbox{(by (\ref{eq:sumpsir}))}\\
& = &  \dis \frac{(p-1)^k}{|H^k|}\cdot \frac{1}{|\sym_k|}\sum_{\pi \in \sym_k} \chi_{\beta}(\pi) \overline{\chi_{\alpha}(\pi)}\\
& = & \langle \chi_{\beta} , \, \chi_{\alpha} \rangle_{\sym_k}\\
& = & \delta_{\alpha,\beta},

\end{array}$$
as claimed.

\end{proof}

To prove our next result on multiplicities, we will use the following, which is an easy corollary of Mackey's Theorem (see \cite[Theorem (5.6)]{isaacs}) and Frobenius Reciprocity.

\begin{theorem}[Mackey]
\label{thm:mackey}
Let ${\cal K}, \, {\cal H} \leq {\cal G}$ be finite groups, and $x_1, \, \ldots , \, x_m$ be representatives for the $({\cal H},{\cal K})$-double cosets in ${\cal G}$ (i.e. ${\cal G}= {\cal H} x_1 {\cal K} \dot \cup  \cdots \dot \cup {\cal H} x_m {\cal K}$). Then, for any class functions $S$ and $T$ of ${\cal H}$ and ${\cal K}$ respectively, we have
$$\langle \Ind^{\cal G}_{\cal H}(S) , \, \Ind^{\cal G}_{\cal K} (T) \rangle_{\cal G}= \dis \sum_{i=1}^m \langle \Res^{{\cal H}^{x_i}}_{{\cal H}^{x_i} \cap {\cal K}}(S^{x_i}) , \, \Res^{\cal K}_{{\cal H}^{x_i} \cap {\cal K}} (T) \rangle_{{\cal H}^{x_i} \cap {\cal K}},$$
where the class function
$S^{x_i}$ of ${\cal H}^{x_i}=x_i^{-1}{\cal H}x_i$ is defined by $S^{x_i}(u)=S(x_iux_i^{-1})$ for all $u \in {\cal H}^{x_i}$ ($1 \leq i \leq m$).

\end{theorem}

We can now prove the following

\begin{theorem}
\label{thm:multiplicities}
For any $i \in I$, any $0 \leq j \leq k$, and any $\alpha \vdash k$, $\beta \vdash j$ and $\ga \vdash k-j$, we have
$$\left\langle \Ind^{G_k}_{H_k} \left( \widetilde{\theta_i^k} \otimes \zeta_{\alpha} \right) , \, \Ind^{G_k}_{G_j \times G_{k-j}} \left( \left( \widetilde{\psi_r^j} \otimes \varphi_{\beta} \right) \bt \left( \widetilde{\psi_i^{k-j}} \otimes \varphi_{\ga} \right) \right) \right\rangle_{G_k} = c^{\alpha}_{\beta, \ga},$$
the Littlewood-Richardson coefficient for the symmetric group $\sym_k$ (see \cite[Theorem 2.8.13]{James-Kerber}).

\end{theorem}
\noindent
{\bf{Remark:}} in the above statement, and in the rest of the paper, we denote some tensor products by $\bt$ instead of $\otimes$ to emphasize the fact that they are outer tensor products.

\begin{proof}

We start by noticing that, if $j=0$ or $j=k$, then $G_j \times G_{k-j}=G_k$. For any $\alpha, \, \beta, \, \gamma \vdash k$, we let, in a natural way, $c^{\alpha}_{\emptyset,\gamma}=\delta_{\alpha,\ga}$ and $c^{\alpha}_{\beta,\emptyset}=\delta_{\alpha,\beta}$. Then, if $j=0$,  the claim becomes $$\left\langle \Ind^{G_k}_{H_k} \left( \widetilde{\theta_i^k} \otimes \zeta_{\alpha} \right) , \,  \widetilde{\psi_i^{k}} \otimes \varphi_{\ga}  \right\rangle_{G_k} = c^{\alpha}_{\emptyset, \ga}=\delta_{\alpha,\ga},$$
which is true for any $\alpha, \,  \gamma \vdash k$ by Theorem \ref{thm:Compatible} and Frobenius Reciprocity. If, on the other hand, $j=k$, then the claim becomes 
$$\left\langle \Ind^{G_k}_{H_k} \left( \widetilde{\theta_i^k} \otimes \zeta_{\alpha} \right) , \, \widetilde{\psi_r^j} \otimes \varphi_{\beta}\right\rangle_{G_k} = c^{\alpha}_{\beta, \emptyset}=\delta_{\alpha,\beta},$$
which is true for any $\alpha, \,  \beta \vdash k$ by Lemma \ref{lem:multiplicitythetaibeta} and Frobenius Reciprocity.

\noindent
From now on, we therefore fix any $0 < j < k$. Take any $\alpha \vdash k$, $\beta \vdash j$ and $\ga \vdash k-j$, and let
$$C_{i,j,\alpha,\beta,\ga}=\left\langle \Ind^{G_k}_{H_k} \left( \widetilde{\theta_i^k} \otimes \zeta_{\alpha} \right) , \, \Ind^{G_k}_{G_j \times G_{k-j}} \left( \left( \widetilde{\psi_r^j} \otimes \varphi_{\beta} \right) \bt \left( \widetilde{\psi_i^{k-j}} \otimes \varphi_{\ga} \right) \right) \right\rangle_{G_k}.$$
We will apply Theorem \ref{thm:mackey} to the groups ${\cal G}=G_k$, ${\cal H}=H_k$ and ${\cal K}=G_j \times G_{k-j}$. Since $H_k$ contains (a copy of) $\sym_k$, which itself contains representatives for the left cosets of $G_j \times G_{k-j}$ in $G_k$ (which are the same as representatives for the left cosets of $\sym_j \times \sym_{k-j}$ in $\sym_k$), there is a single $(H_k, G_j \times G_{k-j})$-double coset in $G_k$. Thus we have $G_k=H_k \cdot (G_j \times G_{k-j})$ and, with the notations of Theorem \ref{thm:mackey}, $m=1$ and $x_1=1$. Also, $H^1_k \cap (G_j \times G_{k-j})=H_k \cap (G_j \times G_{k-j})=H_j \times H_{k-j}$. Hence, by Theorem \ref{thm:mackey}, we have
$$C_{i,j,\alpha,\beta,\ga}=\langle \Res^{H_k}_{H_j \times H_{k-j}} ( \widetilde{\theta_i^k} \otimes \zeta_{\alpha} ) , \, \Res^{G_j \times G_{k-j}}_{H_j \times H_{k-j}} ( (\widetilde{\psi_r^j} \otimes \varphi_{\beta})  \bt ( \widetilde{\psi_i^{k-j}} \otimes \varphi_{\ga})  ) \rangle_{H_j \times H_{k-j}}.$$

Now, for any $(f; \, \pi) \in H_j \times H_{k-j}$, we can write $(f; \, \pi) =(f^{(j)}; \, \pi^{(j)}) \otimes (f^{(k-j)}; \, \pi^{(k-j)})$ in such a way that
$$\begin{array}{rcl} \widetilde{\theta_i^k}(f; \, \pi) & = & \dis \prod_{\nu=1}^{c(\pi)} \theta_i (g_{\nu}(f; \, \pi)) \\
 & = &  \dis \prod_{\nu=1}^{c(\pi^{(j)})} \theta_i (g_{\nu}(f^{(j)}; \, \pi^{(j)})) \cdot \prod_{\nu=1}^{c(\pi^{(k-j)})} \theta_i (g_{\nu}(f^{(k-j)}; \, \pi^{(k-j)})) \\
  & = &  \widetilde{\theta_i^j} (f^{(j)}; \, \pi^{(j)}) \cdot \widetilde{\theta_i^{k-j}} (f^{(k-j)}; \, \pi^{(k-j)})
\end{array}$$
(we only have to be careful, when seeing $\pi= \pi^{(j)} * \pi^{(k-j)}$ as an element of $\sym_j \times \sym_{k-j}$, that $\sym_j $ and
$\sym_{k-j}$ do act on the indices we want).

\noindent
This shows that
\begin{equation}
\label{eq:restheta}
 \Res^{H_k}_{H_j \times H_{k-j}} ( \widetilde{\theta_i^k})=\widetilde{\theta_i^j} \bt \widetilde{\theta_i^{k-j}} .
\end{equation}

Now, by the Littlewood-Richardson Rule (see \cite[Theorem 2.8.13]{James-Kerber}), we have
\begin{equation}
\label{eq:LR}
\Res^{\sym_k}_{\sym_j \times \sym_{k-j}} (\chi_{\alpha})= \dis \sum_{\mu \vdash j \atop \nu \vdash k-j} c^{\alpha}_{\mu, \nu} \cdot \chi_{\mu} \bt \chi_{\nu}.
\end{equation}
Since $\langle \Res^{H_k}_{H_j \times H_{j-k}}(\zeta_{\alpha}), \, \zeta_{\mu} \bt \zeta_{\nu} \rangle_{H_j \times H_{k-j}} = \langle \Res^{\sym_k}_{\sym_j \times \sym_{k-j}} (\chi_{\alpha}), \, \chi_{\mu} \bt \chi_{\nu} \rangle_{\sym_j \times \sym_{k-j}}$ and $\zeta_{\alpha}(1)=\chi_{\alpha}(1)$, this easily gives
\begin{equation}
\label{eq:reszeta}
\Res^{H_k}_{H_j \times H_{j-k}}(\zeta_{\alpha}) = \dis \sum_{\mu \vdash j \atop \nu \vdash k-j} c^{\alpha}_{\mu, \nu} \cdot \zeta_{\mu} \bt \zeta_{\nu}.
\end{equation}
Using (\ref{eq:restheta}) and (\ref{eq:reszeta}), we obtain
\begin{equation}
\label{eq:reslhs}
\Res^{H_k}_{H_j \times H_{k-j}} ( \widetilde{\theta_i^k} \otimes \zeta_{\alpha} ) =  \dis \sum_{\mu \vdash j \atop \nu \vdash k-j} c^{\alpha}_{\mu, \nu} \cdot (\widetilde{\theta_i^j} \otimes \zeta_{\mu}) \bt (\widetilde{\theta_i^{k-j}} \otimes \zeta_{\nu}).
\end{equation}
We also have
\begin{equation}
\label{eq:resrhs}
\Res^{G_j \times G_{k-j}}_{H_j \times H_{k-j}} ( (\widetilde{\psi_r^j} \otimes \varphi_{\beta})  \bt ( \widetilde{\psi_i^{k-j}} \otimes \varphi_{\ga}) ) =\Res^{G_j }_{H_j } ( \widetilde{\psi_r^j} \otimes \varphi_{\beta} )  \bt  \Res^{G_{k-j}}_{H_{k-j}} ( \widetilde{\psi_i^{k-j}} \otimes \varphi_{\ga} ).
\end{equation}
Together, (\ref{eq:reslhs}) and (\ref{eq:resrhs}) yield that $C_{i,j,\alpha,\beta,\ga}$ is equal to
$$\langle \dis \sum_{\mu \vdash j \atop \nu \vdash k-j} c^{\alpha}_{\mu, \nu} \cdot (\widetilde{\theta_i^j} \otimes \zeta_{\mu}) \bt (\widetilde{\theta_i^{k-j}} \otimes \zeta_{\nu}) , \, \Res^{G_j }_{H_j } ( \widetilde{\psi_r^j} \otimes \varphi_{\beta} )  \bt  \Res^{G_{k-j}}_{H_{k-j}} ( \widetilde{\psi_i^{k-j}} \otimes \varphi_{\ga} ) \rangle_{H_j \times H_{k-j}},$$
which in turn is the same as
\begin{equation}
\label{eq:C}
\dis \sum_{\mu \vdash j \atop \nu \vdash k-j} c^{\alpha}_{\mu, \nu} \cdot \langle \widetilde{\theta_i^j} \otimes \zeta_{\mu} , \Res^{G_j }_{H_j } ( \widetilde{\psi_r^j} \otimes \varphi_{\beta} )  \rangle_{H_j} \cdot \langle \widetilde{\theta_i^{k-j}} \otimes \zeta_{\nu} , \Res^{G_{k-j}}_{H_{k-j}} ( \widetilde{\psi_i^{k-j}} \otimes \varphi_{\ga} ) \rangle_{H_{k-j}}.
\end{equation}

\noindent
Now, by Lemma \ref{lem:multiplicitythetaibeta}, we have, for all $\mu \vdash j$,
\begin{equation}
\label{eq:C1}
 \langle \widetilde{\theta_i^j} \otimes \zeta_{\mu} , \Res^{G_j }_{H_j } ( \widetilde{\psi_r^j} \otimes \varphi_{\beta} )  \rangle_{H_j}=\delta_{\mu,\beta}.
 \end{equation}
Also, by Lemma \ref{lem:easyrestrictions}, we have
$$\Res^{G_{k-j}}_{H_{k-j}} ( \widetilde{\psi_i^{k-j}} \otimes \varphi_{\ga} ) = \Res^{G_{k-j}}_{H_{k-j}} ( \widetilde{\psi_i^{k-j}}) \otimes \Res^{G_{k-j}}_{H_{k-j}} ( \varphi_{\ga} ) = \widetilde{\theta_i^{k-j}} \otimes \zeta_{\ga}.$$
Hence, for all $\nu \vdash k-j$, we have
\begin{equation}
\label{eq:C2}
\langle \widetilde{\theta_i^{k-j}} \otimes \zeta_{\nu} , \Res^{G_{k-j}}_{H_{k-j}} ( \widetilde{\psi_i^{k-j}} \otimes \varphi_{\ga} ) \rangle_{H_{k-j}} = \langle \widetilde{\theta_i^{k-j}} \otimes \zeta_{\nu} , \widetilde{\theta_i^{k-j}} \otimes \zeta_{\ga}\rangle_{H_{k-j}}=\delta_{\nu,\ga}
\end{equation}
(since both are irreducible characters of $H_{k-j}$, the same if and only if $\nu = \ga$).

\noindent
Finally, using (\ref{eq:C}), (\ref{eq:C1}) and (\ref{eq:C2}), we obtain 
$$C_{i,j,\alpha,\beta,\ga} = \dis \sum_{\mu \vdash j \atop \nu \vdash k-j} c^{\alpha}_{\mu, \nu} \cdot \delta_{\mu,\beta} \cdot  \delta_{\nu,\ga} = c^{\alpha}_{\beta, \ga},$$
as claimed.
\end{proof}

\medskip
\noindent
We can now finally state and prove the main result of this section.

\begin{theorem}
\label{thm:inducedthetaialpha}
For any $i \in I$, any integer $k \geq 1$ and $\alpha \vdash k$, we have
$$\Ind^{G_k}_{H_k}\left(\widetilde{\theta_i^k} \otimes \zeta_{\alpha}\right)= \dis \sum_{j=0}^k \sum_{\beta \vdash j \atop \ga \vdash k-j} c^{\alpha}_{\beta, \ga} \cdot \Ind^{G_k}_{G_j \times G_{k-j}} \left( (\widetilde{\psi_r^j} \otimes \varphi_{\beta} ) \bt ( \widetilde{\psi_i^{k-j}} \otimes \varphi_{\ga} ) \right).$$

\end{theorem}

\begin{proof}
Since the characters appearing on the right-hand side are pairwise distinct irreducible characters of $G_k$, Theorem \ref{thm:multiplicities} shows that all we have to prove is that the left-hand side and right-hand side characters have the same degree.

\noindent
For the left-hand side, we easily see that 
$$\Ind^{G_k}_{H_k}\left(\widetilde{\theta_i^k} \otimes \zeta_{\alpha}\right)(1) = [ G_k \colon H_k] \cdot  (\widetilde{\theta_i^k} \otimes \zeta_{\alpha})(1)= [ G_k \colon H_k] \cdot \zeta_{\alpha}(1)=p^k \cdot  \chi_{\alpha}(1).$$

\noindent
Now, for any $0 \leq j \leq k$ and any $\beta \vdash j $ and $\ga \vdash k-j$, we have
$$\begin{array}{rcl} \Delta_{j,\beta,\ga} & := & \Ind^{G_k}_{G_j \times G_{k-j}} \left( (\widetilde{\psi_r^j} \otimes \varphi_{\beta} ) \bt ( \widetilde{\psi_i^{k-j}} \otimes \varphi_{\ga} ) \right)(1) \\ & = &  [G_k \colon G_j \times G_{k-j}] \cdot \widetilde{\psi_r^j} (1) \cdot \varphi_{\beta} (1) \cdot  \widetilde{\psi_i^{k-j}}(1) \cdot \varphi_{\ga}(1 )  \\ & = & \dis \frac{k!}{j! \cdot (k-j)!} \cdot  (\psi_r(1))^j \cdot \varphi_{\beta} (1)\cdot \varphi_{\ga}(1 ) \\ & = & \left( {k \atop j}\right) \cdot (p-1)^j \cdot \chi_{\beta}(1) \cdot \chi_{\ga}(1).\end{array}$$
We therefore obtain, for the right-hand side,
$$\begin{array}{rcl} \Delta_k & := &  \dis \sum_{j=0}^k \sum_{\beta \vdash j \atop \ga \vdash k-j} c^{\alpha}_{\beta, \ga} \cdot \Ind^{G_k}_{G_j \times G_{k-j}} \left( (\widetilde{\psi_r^j} \otimes \varphi_{\beta} ) \bt ( \widetilde{\psi_i^{k-j}} \otimes \varphi_{\ga} ) \right)(1)\\
 & = & \dis \sum_{j=0}^k \sum_{\beta \vdash j \atop \ga \vdash k-j} c^{\alpha}_{\beta, \ga} \cdot \Delta_{j,\beta,\ga}\\
 & = & \dis \sum_{j=0}^k \sum_{\beta \vdash j \atop \ga \vdash k-j} c^{\alpha}_{\beta, \ga} \cdot \left( {k \atop j}\right) \cdot (p-1)^j \cdot \chi_{\beta}(1) \cdot \chi_{\ga}(1)\\
 & = &  \dis \sum_{j=0}^k \left( {k \atop j}\right) \cdot (p-1)^j \cdot \left(\sum_{\beta \vdash j \atop \ga \vdash k-j} c^{\alpha}_{\beta, \ga} \cdot \chi_{\beta}(1) \cdot \chi_{\ga}(1)\right)\\
 & = & \left(\dis \sum_{j=0}^k \left( {k \atop j}\right) \cdot (p-1)^j \right)\cdot \chi_{\alpha}(1) \qquad \mbox{(by (\ref{eq:LR}))}\\
 & & \\
 & = & p^k \cdot \chi_{\alpha}(1)

\end{array}$$
(since $p^k=((p-1)+1)^k=\dis \sum_{j=0}^k \left( {k \atop j}\right) \cdot (p-1)^j$). This concludes the proof.

\end{proof}

\section{Main result}
\label{sec:Coef}
We can now state and prove our main result (Theorem \ref{thm:main}). Take any positive integer $w$. Recall that the coefficients we wish to find are the multiplicities of the irreducible characters of $H_w$ in the restriction to $H_w$ of any irreducible character of $G_w$. Equivalently, we want to decompose into irreducibles of $G_w$ the induced character $\Ind_{H_w}^{G_w}(\xi)$ of any $\xi \in  \Irr(H_w)$. Hence we take any $\alpha=(\alpha^1, \, \ldots , \, \alpha^{r-1}, \, \alpha^{r+1} , \, \ldots , \,  \alpha^p) \Vdash w$, and consider $\xi^{\alpha} \in \Irr(H_w)$. Recall that $\xi^{\alpha}$ is given by
$$\xi^{\alpha}= \Ind_{\prod_{i \in I} H_{|\alpha^i|}}^{H_w} (\dis \prod_{i \in I} \widetilde{\theta_i^{|\alpha^i|}}
\otimes \zeta_{\alpha^i}),$$or, letting $H_{\alpha}=\prod_{i \in I} H_{|\alpha^i|}$ and $\Theta^{\alpha}=\prod_{i \in I} \widetilde{\theta_i^{|\alpha^i|}}\otimes \zeta_{\alpha^i}$, by $\xi^{\alpha}=\Ind_{ H_{\alpha}}^{H_w} (\Theta^{\alpha})$.

\medskip
\noindent
The diagram below explains our strategy.

\medskip

\begin{center}
\begin{tikzpicture}

\draw (0,0) node{$G_{\alpha}=\prod_{i \in I} G_{|\alpha^i|}$};
\draw[very thick,->] (0,0.5) -- (0,1.5);
\draw (-2.5,2) node{$\Ind_{H_w}^{G_w}(\xi_{\alpha})$};
\draw (7.5,2) node{$\xi^{\alpha}=\Ind_{ H_{\alpha}}^{H_w} (\Theta^{\alpha})$};
\draw (0,2) node{$G_w$};
\draw[very thick,<-] (1.7,0) -- (3.5,0);
\draw (5,0) node{$ H_{\alpha}=\prod_{i \in I} H_{|\alpha^i|}$};
\draw[very thick,->] (5,0.5) -- (5,1.5);
\draw (7.5,0.1) node{$\Theta^{\alpha}$};
\draw (-2.5,0.1) node{$\widehat{\Theta^{\alpha}}$};
\draw (5,2) node{$H_w$};
\draw[very thick,<-] (1.5,2) -- (4,2);

\end{tikzpicture}

\end{center}
\medskip
\noindent
Starting from $\Theta^{\alpha} \in \Irr(H_{\alpha})$, instead of going right around the above diagram to compute $\Ind_{H_w}^{G_w}(\Ind_{ H_{\alpha}}^{H_w} (\Theta^{\alpha}))$, we will go left to compute (the same character) $\Ind_{G_{\alpha}}^{G_w}(\Ind_{ H_{\alpha}}^{G_{\alpha}} (\Theta^{\alpha}))$. Now, since $H_{|\alpha^i|} \leq G_{|\alpha^i|}$ for all $i \in I$, we have
$$\widehat{\Theta^{\alpha}}:=\Ind_{ H_{\alpha}}^{G_{\alpha}} (\Theta^{\alpha}) = \Ind_{\prod_{i \in I} H_{|\alpha^i|}}^{\prod_{i \in I} G_{|\alpha^i|}} \left( \dis \prod_{i \in I} \widetilde{\theta_i^{|\alpha^i|}}\otimes \zeta_{\alpha^i} \right)=  \dis \prod_{i \in I} \Ind_{H_{|\alpha^i|}}^{G_{|\alpha^i|}} \left(\widetilde{\theta_i^{|\alpha^i|}}\otimes \zeta_{\alpha^i} \right).$$
By Theorem \ref{thm:inducedthetaialpha}, we therefore obtain
$$ \begin{array}{rcl} \widehat{\Theta^{\alpha}}  & = & \dis \prod_{i \in I} \Ind_{H_{|\alpha^i|}}^{G_{|\alpha^i|}} \left(\widetilde{\theta_i^{|\alpha^i|}}\otimes \zeta_{\alpha^i} \right) \\
& = &
\dis \prod_{i \in I} \sum_{j=0}^{|\alpha^i|} \sum_{\beta \vdash j  \atop \ga \vdash |\alpha^i|-j} c^{\alpha^i}_{\beta, \ga} \cdot \Ind^{G_|\alpha^i|}_{G_j \times G_{|\alpha^i|-j}} \left( (\widetilde{\psi_r^j} \otimes \varphi_{\beta} ) \bt ( \widetilde{\psi_i^{|\alpha^i|-j}} \otimes \varphi_{\ga} ) \right)\\
& = & \dis \sum_{0 \leq j_i \leq |\alpha^i| \atop (i \in I)} \sum_{\beta^i \vdash j_i  \atop {\ga^i \vdash |\alpha^i|-j_i \atop (i \in I)}} \prod_{i \in I}c^{\alpha^i}_{\beta^i, \ga^i} \cdot \Ind^{G_|\alpha^i|}_{G_{j_i} \times G_{|\alpha^i|-j_i}} ( (\widetilde{\psi_r^{j_i}} \otimes \varphi_{\beta^i} ) \bt ( \widetilde{\psi_i^{|\alpha^i|-j_i}} \otimes \varphi_{\ga^i} ) )\\
& =: & \dis \sum_{0 \leq j_i \leq |\alpha^i| \atop (i \in I)} \sum_{\beta^i \vdash j_i  \atop {\ga^i \vdash |\alpha^i|-j_i \atop (i \in I)}} \dis \left(\prod_{i \in I}c^{\alpha^i}_{\beta^i, \ga^i} \right) \cdot \widehat{\Theta_{\alpha,\beta,\ga}} ,

\end{array}$$
where we have

$$ \begin{array}{rcl} \widehat{\Theta_{\alpha,\beta,\ga}} & = &  \dis \prod_{i \in I}\Ind^{G_|\alpha^i|}_{G_{j_i} \times G_{|\alpha^i|-j_i}} ( (\widetilde{\psi_r^{j_i}} \otimes \varphi_{\beta^i} ) \bt ( \widetilde{\psi_i^{|\alpha^i|-j_i}} \otimes \varphi_{\ga^i} ) ) \\
& = & \dis  \Ind^{\prod_{i\in I}G_|\alpha^i|}_{\prod_{i\in I}G_{j_i} \times G_{|\alpha^i|-j_i}} \left(\prod_{i \in I} (\widetilde{\psi_r^{j_i}} \otimes \varphi_{\beta^i} ) \bt ( \widetilde{\psi_i^{|\alpha^i|-j_i}} \otimes \varphi_{\ga^i} ) \right).\\
\end{array}$$
Inducing to $G_w$, we now have 
$$\Ind_{H_w}^{G_w}(\xi_{\alpha}) = \Ind_{G_{\alpha}}^{G_w}(\widehat{\Theta^{\alpha}}) = \dis \sum_{0 \leq j_i \leq |\alpha^i| \atop (i \in I)} \sum_{\beta^i \vdash j_i  \atop {\ga^i \vdash |\alpha^i|-j_i \atop (i \in I)}} \dis \left(\prod_{i \in I}c^{\alpha^i}_{\beta^i, \ga^i} \right) \Ind_{G_{\alpha}}^{G_w}(\widehat{\Theta_{\alpha,\beta,\ga}}) ,$$
where, if we write $J=(j_1, \, \ldots , j_{r-1}, \, j_{r+1}, \, \ldots, \, j_p)$, $|J|=\sum_{i \in I} j_i$ and $\alpha-J=(\alpha^1-j_1, \, \ldots  , \, \alpha^{r-1}- j_{r-1}, \, \alpha^{r+1}- j_{r+1}, \, \ldots, \, \alpha^p-j_p)$, we have 
 $$\begin{array}{l} \Ind_{G_{\alpha}}^{G_w}(\widehat{\Theta_{\alpha,\beta,\ga}})  =  \dis  \Ind^{G_w}_{\prod_{i\in I}G_{j_i} \times G_{|\alpha^i|-j_i}} \left(\prod_{i \in I} (\widetilde{\psi_r^{j_i}} \otimes \varphi_{\beta^i} ) \bt ( \widetilde{\psi_i^{|\alpha^i|-j_i}} \otimes \varphi_{\ga^i} ) \right) \\
  =  \dis  \Ind^{G_w}_{\prod_{i\in I}G_{j_i} \times \prod_{i\in I}G_{|\alpha^i|-j_i}} \left(\prod_{i \in I} (\widetilde{\psi_r^{j_i}} \otimes \varphi_{\beta^i} ) \bt \prod_{i\in I}( \widetilde{\psi_i^{|\alpha^i|-j_i}} \otimes \varphi_{\ga^i} ) \right) \\
  =  \dis  \Ind^{G_w}_{G_J\times G_{\alpha-J}} \left(\prod_{i \in I} (\widetilde{\psi_r^{j_i}} \otimes \varphi_{\beta^i} ) \bt \prod_{i\in I}( \widetilde{\psi_i^{|\alpha^i|-j_i}} \otimes \varphi_{\ga^i} ) \right) \\
 
= \dis \Ind^{G_w}_{G_{|J|} \times G_{\alpha-J}}  \left( \Ind^{G_{|J|} \times G_{\alpha-J}}_{G_{J} \times G_{\alpha-J}} \left(\prod_{i \in I} (\widetilde{\psi_r^{j_i}} \otimes \varphi_{\beta^i} ) \bt \prod_{i\in I}( \widetilde{\psi_i^{|\alpha^i|-j_i}} \otimes \varphi_{\ga^i} ) \right) \right) \\

= \dis \Ind^{G_w}_{G_{|J|} \times G_{\alpha-J}} \left( \Ind^{G_{|J|}}_{G_{J}} \left(\prod_{i \in I} (\widetilde{\psi_r^{j_i}} \otimes \varphi_{\beta^i} ) \right) \bt \prod_{i\in I}( \widetilde{\psi_i^{|\alpha^i|-j_i}} \otimes \varphi_{\ga^i} ) \right) .

 \end{array}$$
 And, iterating \cite[Proposition 4.1]{Stein}, we have
 $$ \Ind^{G_{|J|}}_{G_{J}} \left(\prod_{i \in I} (\widetilde{\psi_r^{j_i}} \otimes \varphi_{\beta^i} ) \right)= \dis \sum_{\ga^r \vdash |J|} c_{(\beta^i, \, i \in I)}^{\ga^r} \cdot (\widetilde{\psi_r^{|J|}} \otimes \varphi_{\ga^r} ),$$
 where $c_{(\beta^i, \, i \in I)}^{\ga^r}=\langle \chi_{\ga^r}, \, \Ind_{\prod_{i \in I} \sym_{j_i}}^{\sym_{|J|}} (\prod_{i \in I}\chi_{\beta^i}) \rangle_{\sym_{|J|}}$ is the coefficient obtained by iterating the Littlewood-Richardson Rule. 
 
 \noindent
 We therefore obtain
 $$\begin{array}{l} \Ind_{G_{\alpha}}^{G_w}(\widehat{\Theta_{\alpha,\beta,\ga}})   =  \dis \sum_{\ga^r \vdash |J|} c_{(\beta^i, \, i \in I)}^{\ga^r} \cdot \Ind^{G_w}_{G_{|J|} \times G_{\alpha-J}} \left( (\widetilde{\psi_r^{|J|}} \otimes \varphi_{\ga^r} )  \bt \prod_{i\in I}( \widetilde{\psi_i^{|\alpha^i|-j_i}} \otimes \varphi_{\ga^i} ) \right)\\
 =  \dis \sum_{\ga^r \vdash |J|} c_{(\beta^i, \, i \in I)}^{\ga^r} \cdot \Ind^{G_w}_{G_{|\ga^r|} \times \prod_{i \in I}G_{|\ga^i|}} \left( (\widetilde{\psi_r^{|\ga^r|}} \otimes \varphi_{\ga^r} )  \bt \prod_{i\in I}( \widetilde{\psi_i^{|\ga^i|}} \otimes \varphi_{\ga^i} ) \right),\\
 
 \end{array}$$
 whence 
 $$\Ind_{G_{\alpha}}^{G_w}(\widehat{\Theta_{\alpha,\beta,\ga}}) =\dis \sum_{\ga^r \vdash |J|} c_{(\beta^i, \, i \in I)}^{\ga^r} \cdot \chi^{(\ga^1, \, \ldots , \, \ga^p)}.$$

\noindent
Finally, this yields
$$\Ind_{H_w}^{G_w}(\xi^{\alpha})=\dis \sum_{0 \leq j_i \leq |\alpha^i| \atop (i \in I)} \sum_{\beta^i \vdash j_i  \atop {\ga^i \vdash |\alpha^i|-j_i \atop (i \in I)}} \sum_{\ga^r \vdash \sum_{i \in I} j_i}  \left(\prod_{i \in I}c^{\alpha^i}_{\beta^i, \ga^i} \right) \cdot c_{(\beta^i, \, i \in I)}^{\ga^r} \cdot \chi^{(\ga^1, \, \ldots , \, \ga^p)},$$
 which is the decomposition of $\Ind_{H_w}^{G_w}(\xi^{\alpha})$ into irreducible characters of $G_w$. Note that each irreducible $\chi^{(\ga^1, \, \ldots , \, \ga^p)}$ appears (with multiplicity) several times, corresponding to the choice of the $\beta^i$'s ($i \in I$). This can now be rewritten as

\begin{theorem}
\label{thm:main}
For any integer $w>0$ and $\alpha=(\alpha^1, \, \ldots , \, \alpha^{r-1}, \, \alpha^{r+1} , \, \ldots , \,  \alpha^p) \Vdash w$, we have
$$\Ind_{H_w}^{G_w}(\xi^{\alpha})=\dis \sum_{\ga=(\ga^1, \, \ldots , \, \ga^p) \Vdash w} k_{\alpha,\ga} \cdot \chi^{\ga},$$
where, for any $\ga=(\ga^1, \, \ldots , \, \ga^p) \Vdash w$,
$$k_{\alpha,\ga} = \dis \sum_{\beta^i \vdash |\alpha^i|-|\ga^i| \atop (i \in I)} \left(\prod_{i \in I}c^{\alpha^i}_{\beta^i, \ga^i} \right) \cdot c_{(\beta^i, \, i \in I)}^{\ga^r}.$$
In particular, $k_{\alpha,\ga} =0$ unless $|\ga^i| \leq |\alpha^i|$ for all $i \in I$.

\end{theorem}

\begin{remark}
We recover from Theorem \ref{thm:main} the fact that our basic set $\{ \chi^{\ga} \in \Irr(G_w) \, | \, \ga^r=\emptyset \}$ corresponds to $\Irr(H_w)$ (see Theorem \ref{thm:Compatible} and Remark \ref{rem:Compatible}). Indeed, for $\xi^{\alpha}$ to appear in $\Res^{G_w}_{H_w}(\chi^{\ga})$, we must have $|\alpha^i| \geq |\ga^i| $ for all $i \in I$. But, if $\ga^r=\emptyset$, then we already have $\dis \sum_{i \in I} | \ga_i|=\sum_{i=1}^p | \ga_i|=w=\sum_{i \in I} | \alpha_i|$, so that we can only have $|\alpha^i| = |\ga^i| $ for all $i \in I$. We then get $k_{\alpha,\ga}= \dis \left( \prod_{i \in I}c^{\alpha^i}_{\emptyset, \ga^i} \right) \cdot c_{(\emptyset, \, \ldots , \, \emptyset)}^{\emptyset}=\prod_{i \in I} \delta_{\alpha^i,\ga^i}$, so that $k_{\alpha,\ga}=\delta_{\hat{\alpha}, \ga}$ (with the notation of Theorem \ref{thm:Compatible}).

\end{remark}

 \bibliographystyle{abbrv}
\bibliography{references}

\end{document}